\RequirePackage[final]{graphicx}
\documentclass[11pt, oneside, dvipsnames, svgnames, table, final]{amsart}

\usepackage{fourier}
\usepackage[margin=2.3cm]{geometry}

\usepackage{tikz}
\usetikzlibrary{matrix,arrows, calc, patterns}
\usepackage[outline]{contour}

\definecolor{color1}{RGB}{134,14,156}
\definecolor{color2}{RGB}{6,138,39}
\definecolor{color3}{RGB}{12,39,156}
\definecolor{color4}{RGB}{153,0,0}
\definecolor{mygray}{RGB}{192, 192, 192}

\usepackage[foot]{amsaddr}
\usepackage{adjustbox}
\usepackage{etoolbox}
\usepackage{ifdraft}
\usepackage{xcolor}

\usepackage[colorlinks=true,citecolor=green!50!black, linkcolor=red!50!black, pagebackref=false, final]{hyperref}

\ifdraft{\usepackage[color,notref, notcite]{showkeys}}{}
\usepackage{xargs}
\usepackage{amsmath, amsthm, mathrsfs, amssymb}
\usepackage{thmtools}
\usepackage{thm-restate}
\usepackage{tikz-cd} 
\usepackage[final]{graphicx}
\usepackage{enumitem}
\usepackage{bm}

\newcommand{\mathbx}[1]{\mathbb{#1}}

\usepackage[colorinlistoftodos,prependcaption,textsize=tiny]{todonotes} 

\theoremstyle{plain}
\ifcscounter{chapter}{%
        \newtheorem{thm}{Theorem}[chapter]%
        }%
        {%
        \newtheorem{thm}{Theorem}[section]%
        }

\newtheorem{lemma}[thm]{Lemma}

\newtheorem{corollary}[thm]{Corollary}

\newtheorem{proposition}[thm]{Proposition}

\newtheorem*{claim*}{Claim} 
\newtheorem*{thm*}{Theorem}
\newtheorem*{theorem*}{Theorem}

\theoremstyle{definition}

\newtheorem{definition}[thm]{Definition}

\newtheorem{notation}[thm]{Notation}

\theoremstyle{remark}

\newtheorem{remark}[thm]{Remark}

\newtheorem{conjecture}[thm]{Conjecture}

\newcommand{\C}{\mathbx{C}}

\newcommand{\Q}{\mathbx{Q}}

\newcommand{\Z}{\mathbx{Z}}

\newcommand{\A}{\mathbx{A}}

\newcommand{\ZZ}{\Z}

\newcommand{\PP}{\mathbx{P}}

\newcommand{\aone}{\A^1}
\newcommand{\Aone}{\aone}

\newcommand{\M}{\mathsf{M}}

\newcommand{\Ext}{\operatorname{Ext}}

\newcommand{\Gl}{\operatorname{GL}}

\newcommand{\Hom}{\operatorname{Hom}}

\newcommand{\cat}[1]{\text{\bf #1}}

\newcommand{\tensor}{\otimes}

\newcommand{\spec}{\operatorname{Spec}}

\newcommand{\Spec}{\spec}

\renewcommand{\P}{\operatorname{P}}

\newcommand{\sh}[1]{\mathcal{#1}}

\newcommand{\weq}{\simeq}

\newcommand{\op}{\text{op}}

\newcommand{\GL}{\Gl}

\newcommand{\Eoh}{\mathrm{E}}

\newcommand{\Hoh}{\mathrm{H}}

\newcommand{\id}{\mathrm{id}}
\newcommand{\isom}{\cong}
\newcommand{\iso}{\isom}

\newcommand{\mto}[1]{\stackrel{#1}{\longrightarrow}}
\newcommand{\isomto}{\mto{\isom}}

\newcommand{\sm}{\setminus}

\newcommand{\coker}{\operatorname{coker}}

\newcommand{\im}{\operatorname{Im}}

\newcommand{\MW}{\mathsf{MW}}

\newcommand{\et}{\textup{\'et}}

\newcommand{\bpi}{\bm{\pi}}

\newcommand{\Ab}{\mathrm{Ab}}

\newcommand{\K}{{{\mathbf K}}}

\DeclareMathOperator{\uHom}{\underline{Hom\kern-.05em}\kern.1em}

\newcommand{\MU}{\mathsf{MU}}

\usepackage{comment}

\usepackage[backref=true,giveninits=true,backend=biber,style=numeric,url=false]{biblatex}
\ifdraft{\pagecolor{green!10}}{}

\newcommand{\BP}{\mathsf{BP}}
\newcommand{\MZ}{\mathsf{M}\Z}
\newcommand{\BPtop}{\mathsf{BP}^{\mathrm{top}}}

\newcommand{\one}{\mathbb{1}}

\newcommand{\Zpinf}{\Z/\left(p^\infty\right)}
\newcommand{\ZIinf}{\Z/\left(I^\infty\right)}
\newcommand{\SpS}{\mathbb{S}}

\author{Sebastian Gant and Ben Williams}
\address{Department of Mathematics, the University of British Columbia\\
  1984 Mathematics Rd \\
  Vancouver BC V6T 1Z2\\
Canada.}
\email[W.~S.~Gant]{wsgant@math.ubc.ca}
\email[B.~Williams]{tbjw@math.ubc.ca}

\begin{document}
\title{Motivic Homotopy Groups of Spheres and Free Summands of Stably Free Modules}

\begin{abstract}
  Working over an algebraically closed field $k$ of characteristic $0$, we show that the motivic stable homotopy groups of the sphere spectrum can be determined entirely from the motivic homotopy groups of the $p$-completed sphere spectra and the motivic cohomology of the ground field, except possibly for the $0$ and $-1$-stems. Using this, we show that the complex realization maps from the motivic homotopy groups to the classical stable homotopy groups are isomorphisms in a range of bidegrees. We apply this to deduce that complex realization also induces isomorphisms on unstable homotopy groups for Stiefel varieties $V_r(\A^n_k)$ in a range of bidegrees. We use this to determine when the projection map $V_r(\A^n_k) \to V_1(\A^n_k)$ admits a right inverse, settling the question of when the universal stably-free module of type $(n,n-1)$ admits a free summand of given rank.
\end{abstract}
\thanks{We acknowledge the support of the Natural Sciences and Engineering Research Council of Canada (NSERC), RGPIN-2021-02603. We acknowledge that this research was supported in part by the Pacific Institute for the Mathematical Sciences (PIMS) under CRG 41.}

\maketitle

\section{Introduction}
\label{sec:introduction}

Let $R$ be a ring, assumed unital and commutative, and $P$ an $R$-module. One says that $P$ is \emph{stably free of type $(n,r)$} if there is an isomorphism
\[ P \oplus R^{n-r} \iso R^n. \]
In \cite{Gant2025}, we proved that a stably free module $P$ of type $(24m,24m-1)$ admits a free summand of rank $2$ if $R$ contains a field $k$ of characteristic $0$, and of rank $3$ if that field has at most one quadratic extension. The method is to convert the problem to motivic homotopy theory: specifically, whether a realization map relating the motivic and classical homotopy groups of Stiefel varieties is injective. This problem is then solved by using recent excellent progress on the calculation of the motivic homotopy sheaves of spheres---notably, \cite{Rondigs2019} and \cite{Rondigs2024}. Their calculations are of a general nature, valid over any field, but limited in the the bidegrees to which they apply.

Let $k=\bar k$ be an algebraically closed field of characteristic $0$. A great deal more is known about the motivic homotopy groups over $k$ than in general, using the motivic Adams and Adams--Novikov spectral sequences: \cite{Dugger2010, Isaksen2019, Isaksen2022, Isaksen2020, Stahn2021} and other work. The apparent obstacle to using these calculations in our work is that these spectral sequences calculate the homotopy groups of $p$-completed objects, rather than of the sphere spectrum itself. We surmount this obstacle and consequently generalize the results of \cite{Gant2025} to all $r$ over the base field $k$.

We write $\one$ for the motivic sphere spectrum, and $\bpi_{s,w}(\one)$ for the bigraded motivic homotopy sheaf in bigrading $(s,w)$. We use a regular font to denote the group of global sections of this sheaf, so that
\[ \pi_{s,w}(\one) = \bpi_{s,w}(\one)(k).\]
We will also write $\MZ$ for the motivic Eilenberg--MacLane spectrum. This shown to be a commutative ring spectrum in  \cite[Example 3.4]{Dundas2003}. In particular there is a unit map $\one \to \MZ$ inducing maps $\pi_{s,w}(\one) \to \Hoh^{-s}(\Spec k; \Z(-w))$ on homotopy groups, which we also call ``unit maps''. They may also be considered Hurewicz maps if we use $\MZ$ to represent a homology theory: $\Hoh^{-s}(\Spec k; \Z(-w)) = \Hoh_{s}(\Spec k ; \Z(w))$.

The paper establishes four main theorems. Our first concerns the relationship between the sphere spectrum and its $p$-completions. We write $\sh P$ for the set of all prime numbers. 
\begin{restatable}{theorem}{surj}\label{th:surjForSphere}
  Let $s,w$ be integers satisfying $s \neq 0$. Then the comparison map
  \[ \pi_{s,w}(\one) \to \prod_{p \in \sh P} \pi_{s,w}(\one^\wedge_p) \]
  is split surjective, and the kernel is the subgroup of divisible elements. Additionally, if $s \neq -1$, then the kernel is isomorphic to the motivic cohomology group $\Hoh^{-s}(\Spec k; \Z(-w))$ and the inclusion of the kernel is split by the unit map $\pi_{s,w}(\one) \to \Hoh^{-s}(\Spec k; \Z(-w))$.
\end{restatable}

This theorem assures us that the Adams and Adams--Novikov-spectral sequences do, with minor exceptions, calculate the motivic homotopy groups (i.e., the global sections of the homotopy sheaves) of $\one$ over $k$. The kernel $\Hoh^{-s}(\Spec k; \ZZ(-w))$ is either known or conjectured to vanish in a large range of bidegrees.

\begin{notation}
  We write $\SpS$ for the classical sphere spectrum, to distinguish it from the motivic $\one$. This notation is adopted from \cite{Rondigs2024}, and is justified by its usefulness.
\end{notation}

Among the implications of Theorem \ref{th:surjForSphere} is that complex realization induces an isomorphism
\begin{equation}
  \label{eq:3}
 \pi_{s,w}(\one) \isomto \pi_s(\SpS), \text{\quad when $-1 \le w \le \frac{1}{2}s+1$ and $1 \le s$},
\end{equation}
where the target is the ordinary stable homotopy group of the classical sphere spectrum. This is because the maps  $\pi_{s,w}(\one^{\wedge}_p) \to \pi_s(\SpS^{\wedge}_p)$ are isomorphisms in this range; see for instance \cite{Gheorghe2017} when $p=2$ and \cite{Stahn2018} when $p>2$.
We give a fuller account of what we now know about $\pi_{s,w}(\one)$ in Subsection \ref{sec:furth-remarks-homot} below.

There are two ingredients in the proof of Theorem \ref{th:surjForSphere}. First we consider the algebraic properties of the $\Ext$-completions $\Ext(\Z/(p^\infty), \pi_{s,w}(\one))$, which agree with $\pi_{s,w}(\one^{\wedge}_p)$ except when $s=0$. Finiteness of $\prod_{p} \pi_{s,w}(\one^\wedge_p)$ is now sufficient to prove that the completion map belongs in a split short exact sequence 
\[ 0 \to D_{s,w} \to \pi_{s,w}(\one) \to \prod_p \pi_{s,w}(\one^{\wedge}_p) \to 0\]
where $D_{s,w}$ is the subgroup of divisible elements in $\pi_{s,w}(\one)(k)$. Except when $s\in \{0,-1\}$, this kernel is observed to be a $\Q$-vector space. We then rely on \cite{Ananyevskiy2017}, which draws on \cite{Cisinski2019}, for the identification 
\[ \Q \tensor_\Z \pi_{s,w}(\one) \iso \Hoh^{-s}(\Spec k; \Z(-w)) \] 
when $k$ is algebraically closed and $s \not \in \{0, -1\}$.

\begin{remark}\label{rem:Hornbostel}
  The arguments we use here are prefigured by the proof of \cite[Lemma 2.10]{Hornbostel2018}. Although that lemma is concerned with lifting one specific element $\mu_9 \in \pi_{9,5}(\one^\wedge_2)$ to $\pi_{9,5}(\one)$, the proof that is given there is considerably more general and agrees in outline with much of our argument for Theorem \ref{th:surjForSphere}. We are grateful to Jens Hornbostel for pointing this out to us.
\end{remark}

Since our ambition is to establish geometric and algebraic results, we pass from stable to unstable homotopy theory. Our tool here is the motivic Freudenthal suspension theorem of \cite{Asok2024}. The unstable analogs of the isomorphisms \eqref{eq:3} are recorded in Proposition \ref{prop:comparisonOfUH}.

For application to projective modules, the most important motivic spheres are the spheres $S^{2n-1,n} \weq \GL(n)/\GL(n-1)$, which form part of the larger family of Stiefel varieties $V_r(\A^n_k) = \GL(n)/\GL(n-k)$. Here $\GL(n-k)$ is embedded in $\GL(n)$ by block summation.
  \begin{equation}
    \label{eq:10}
    A \mapsto
    \begin{bmatrix}
      I_k & 0 \\ 0 & A
    \end{bmatrix},
  \end{equation}
where $I_k$ is the $k \times k$ identity matrix.
 The Stiefel varieties have the property that they are the base spaces for universal stably free modules, as described in \cite{Raynaud1968a} and in \cite[Sec.~3]{Gant2025}. Specifically, a morphism of schemes
\[ \phi:\Spec R \to V_r(\A_k^n) \]
corresponds to an $R$-module $P$ equipped with an isomorphism $P \oplus R^r \to R^n$. 

Let $U(r)$ denote the unitary group of $r \times r$ matrices. The \emph{complex Stiefel manifolds} are the quotient manifolds $W_r(\C^n) = U(n)/U(n-k)$, where the embeddings follow the same pattern as in \eqref{eq:10}.

Our second and third main theorems constitute a pair, and they deal with the calculation of the unstable homotopy groups of the Stiefel varieties. Before we state these theorems, we discuss some notation.
\begin{notation}
There are two useful bigrading conventions for motivic spheres, and consequently for homotopy sheaves and groups. There is the original convention:
\[ S^{s,w} = S^{s-w} \wedge (\A^1_k \sm \{0\})^{\wedge w}, \]
which we will call the ``stem-weight'' bigrading,
and the convention of \cite{Hu2011b}:
\[ S^{c+w\alpha} = S^c \wedge  (\A^1_k \sm \{0\})^{\wedge w}, \]
which we will call the ``coweight-weight'' bigrading.
Stem-weight bigrading is particularly well suited to motivic Adams and Adams--Novikov spectral sequence calculations over $\C$, and is what is used in the extensive literature on this topic. On the other hand, coweight-weight bigrading is better for use in problems concerning motivic connectivity, for instance, the motivic Freudenthal suspension theorem of \cite{Asok2024}. As a consequence, we use both conventions in this paper. In practice, this means that the notation $c+w\alpha$ should be read as $(c+w,w)$. Here is the same homotopy group written in both conventions:
\[ \pi_{c+w\alpha}(X) = \pi_{c+w,w}(X). \]
The symbol $\alpha$ does not appear with any other sense in this paper.
\end{notation}

Complex realization induces maps of homotopy groups
\[  \pi_{d+e\alpha}(V_r(\A^n)) = \pi_{d+e,e}(V_r(\A^n)) \to \pi_{d+e}(W_r(\C^n)).\]
One may inquire how close these maps are to being isomorphisms. The complexity of our answer means that it is clearer to divide it between two theorems. The first establishes surjectivity of the map in a range of bidegrees, which specializes to establish an isomorphism in some cases.
\begin{restatable}{theorem}{surjectivityStiefel}\label{th:surjectivityStiefel}
  Let $d,e$ be integers and $r,n$ be positive integers. Suppose all the following hold
  \begin{enumerate}
  \item\label{ineq:f} $r \le n-2$;
  \item $d \le 2n-2r-3$;
  \item $e \le d+4-r$;
  \item\label{ineq:z} $2n \le e+d$;
  \end{enumerate}
  The complex realization map
  \[ \pi_{d+e\alpha}(V_r(\A^n)) \to \pi_{d+e}(W_r(\C^n)) \]
  is split surjective, the kernel is the maximal divisible subgroup of the domain, which is uniquely divisible, and the target is a torsion group.

  If additionally $n-1 \le e$, then the realization map is an isomorphism.
\end{restatable}
This has the disadvantage that the admissible $e$-values are starkly constrained when $r$ is large by the inequality $e \le d + 4 -r$. The second theorem establishes injectivity without this constraint.
\begin{restatable}{theorem}{injectivityStiefel} \label{th:injectivityStiefel}
  Let $d,e$ be integers and $r,n$ be positive integers. Suppose all the following hold
  \begin{enumerate}
  \item\label{ineq:fprime} $r \le n-2$;
  \item\label{ineq:needsrn2} $d \le 2n-2r-3$;
  \item $e \le d+3$;
  \item $2n \le e+d$;
  \item\label{ineq:zprime} $n-1 \le e$.
  \end{enumerate}
  The realization map
  \[ \pi_{d+e\alpha}(V_r(\A^n)) \to \pi_{d+e}(W_r(\C^n)) \]
  is injective.
\end{restatable}
The methods of proofs of both theorems are the same, induction on $r$ using homological algebra in the mode of the five lemma, the base case of the spheres being supplied by Proposition \ref{prop:comparisonOfUH}.

Our fourth main theorem concerns the existence of a right inverse (or section) to a projection map $\rho : V_r(\A^n_k) \to V_1(\A^1_k)$. One construction of $\rho$ is as follows: assuming $1 \le r \le n$ are integers, there are canonical embedding homomorphisms $\GL(n-r) \to \GL(n-1) \to \GL(n)$. Consequently there is a canonical map of quotients
\[ \GL(n)/\GL(n-r) =  V_r(\A^n_k)\overset{\rho}\to V_1(\A^1_k) = \GL(n)/\GL(n-1). \]
If one views $V_r(\A^n_k)$ as representing stably free modules, the map $\rho$ corresponds to the construction that takes a module $P$ satisfying $P\oplus R^r \iso R^n$ and produces a module $Q=P\oplus R^{r-1}$, which satisfies $Q \oplus R \iso R^n$. In this guise, the map $\rho$ figures heavily in \cite{Raynaud1968a}, where it is named ``$P$'', a letter we use for a projective module. In \cite[Thm.~6.5]{Raynaud1968a}, the existence of a right inverse for $\rho$ is ruled out except possibly when $n$ is divisible by a certain positive integer $b_r$, the $r$-th James (or Atiyah--Todd) number. Our last main theorem is the converse statement over $\bar \Q$. We state it as a biconditional, but only the ``if'' direction is new.
\begin{restatable}{theorem}{completeSplittingResult}\label{th:completeSplittingResult}
  Let $n$ be a positive integer and suppose $2 \le r \le n-2$. Then
  \[ \rho : V_r(\A_{\bar \Q}^n) \to V_1(\A_{\bar \Q}^n) \]
  has a right inverse if and only if $b_r \mid n$.
\end{restatable}
The restriction from $k$, assumed algebraically closed of characteristic $0$, to $\bar \Q$ is no restriction at all; a right inverse constructed over $\bar \Q$ induces a right inverse over any field $k$ containing an algebraic closure of $\Q$.

The proof relies on work done in \cite[\S 4 and \S 7]{Gant2025}, which we do not repeat in full here. The existence of a right inverse is equivalent to an ostensibly weaker condition: the induced map of homotopy groups
\[\rho_* : \pi_{n-1+n\alpha}(V_r(\A^n_{\bar \Q})) \to \pi_{n-1+n\alpha}(V_1(\A^n_{\bar \Q})) \]
is surjective. The equivalence is a statement that a morphism in the homotopy category can be lifted to a map of schemes: the variety $V_1(\A^n_{\bar \Q})$ is homotopy equivalent to $S^{n-1+n\alpha}$, so that $\pi_{n-1+n\alpha}(X)$ is the set of maps $V_{n-1+n\alpha}(\A^n_{\bar \Q}) \to X$ in the motivic homotopy category.

In particular, any obstruction to the existence of a right inverse is witnessed by a nonzero cokernel for $\rho_*$,  consisting of some elements of $\pi_{n-2 + n\alpha}(V_{r-1}(\A^{n-1}_{\bar \Q}))$. This group is in the range of bidegrees covered by Theorem \ref{th:injectivityStiefel}, so that any obstruction can be detected in the classical group $\pi_{2n-2}(W_{r-1}(\C^{n-1}))$. This allows the reduction of the whole problem to its classical analogue, the question of whether a right inverse exists for the projection $\rho : W_r(\C^n) \to W_1(\C^n)$. This was settled by Adams and Walker in \cite{Adams1965}: the answer being that a right inverse exists whenever $b_r \mid n$.

The new implication in Theorem \ref{th:completeSplittingResult} admits a translation into the language of stably free modules.
\begin{corollary}\label{cor:freeSummands}
  Suppose $R$ is a commutative ring containing an algebraic closure of $\Q$. Let $P$ be an $R$-module for which there exists an isomorphism
  \[ P \oplus R \iso R^n \]
  for some positive integer $n$.
  If $b_r \mid n$, then there exists a decomposition
  \[ P \iso Q \oplus R^{r-1}. \]
\end{corollary}
This corollary follows directly from Theorem \ref{th:completeSplittingResult} using \cite[Prop.~4.4]{Gant2025}.

We sketch the argument used in \cite{Gant2025} for the benefit of the reader who would prefer to keep reading this paper rather than stopping to refer to that one. Since $b_r \mid n$, there is a right-inverse map $\psi: V_1(\A^n_{\bar \Q}) \to V_r(\A^n_{\bar \Q})$. We know that $V_r(\A^n_{\bar Q})$ represents $R$-modules $Q$ equipped with an isomorphism $Q \oplus R^r \iso R^n$, so that the map $\psi$ allows one to take any $P$ satisfying $P \oplus R \iso R^n$ and to produce a corresponding $Q$. The relation $\rho \circ \psi = \id$ implies that $P \iso Q \oplus R^{r-1}$.

\subsection{Outline}
\label{sec:outline}

In Section \ref{sec:completions}, we discuss the Ext-completion $\Ext(\Q/\Z, A)$ of an abelian group $A$. This may be viewed as a completion at all primes. We also handle the variant where only a subset of the prime numbers is considered. We make no claims to originality about this section, but we need it as a reference for our arguments concerning the homotopy groups of $p$-completed motivic spectra.

Section \ref{sec:sphere-spectrum-over} is devoted to applying these technical results to the sphere spectrum, proving Theorem \ref{th:surjForSphere} and collecting some of its immediate consequences, which include Corollary \ref{cor:isoForSphereRealization}, which tells us that complex realization induces an isomorphism of homotopy groups in a range of bidegrees. Subsection \ref{sec:furth-remarks-homot}, which contains no new results, is an incomplete guide to the groups $\pi_{s,w}(\one)$ over an algebraically closed field of characteristic $0$, included  in the hope that the reader may find it useful.

Having proved that the motivic and classical stable homotopy groups of spheres agree in some bidegrees, in Section \ref{sec:unst-homot-groups} we exploit the $\PP^1$-Freudenthal theorem of \cite{Asok2024} and an induction argument to prove that complex realization induces isomorphisms on unstable homotopy groups of motivic spheres and of Stiefel varieties in certain ranges. In Figure \ref{fig:chart}, we give a visual presentation of some of what we know about the unstable homotopy groups of the spheres.

In Section \ref{sec:appl-sect-stief}, we combine what we have proved so far with \cite{Gant2025} to deduce that right inverses for projection maps $V_r(\A_{\bar \Q}^n) \to V_1(\A_{\bar \Q}^n)$ exist if and only if right inverses exist for the corresponding projections $W_r(\C^n) \to W_1(\C^n)$, thereby establishing our last main theorem, Theorem \ref{th:completeSplittingResult}.

In our later results, notably in Proposition \ref{prop:comparisonOfUH} and Theorem \ref{th:surjectivityStiefel},  we are concerned with maps that are isomorphisms modulo uniquely divisible subgroups of the source. We consider the homological algebra of such morphisms in Appendix \ref{sec:homol-algebra-modulo}.

\subsection{The ground field}
\label{sec:ground-field}

Some of our sources, particularly \cite{Dugger2010}, \cite{Gheorghe2017}, \cite{Isaksen2022} and other works on the motivic Adams and Adams–Novikov spectral sequences, work over the ground field $\C$. For the purposes of calculating $\pi_{s,w}(\one^\wedge_p)$, we can use \cite{Rondigs2008a} to replace $\C$ by an arbitrary algebraically closed field $k$ of characteristic $0$. We ultimately wish to apply our calculations, in Section \ref{sec:appl-sect-stief}, in the case where $k = \bar \Q$, which is the most general available to us. It imposes essentially no extra cost to work throughout over a general algebraically closed field $k$ of characteristic $0$.


\subsection{Acknowledgements}
\label{sec:acknowledgments}

This paper was conceived of following a talk by Zhouli Xu at the International Workshop on Algebraic Topology in Hangzhou, China 2025. We thank the organizers of this conference for staging it. We also thank Zhouli Xu for helpful conversations, and Dan Isaksen and Zhouli Xu both for informative correspondence on the nature of the motivic homotopy groups.

We thank an anonymous referee for a thoughtful and constructive reading of the paper.

\section{Ext-Completions}%
\label{sec:completions}

We make no claims about the originality of this section, the results of which are elementary and, in many cases, implicit in \cite{Rotman1968}. It is included here for reference and to set notation. 

\subsection{Definitions}
\label{sec:definitions}

Let $\sh P$ denote the set of all prime numbers. For the rest of this subsection, fix a subset $I \subseteq \sh P$, which may be empty. We say that a positive integer $n$ is \emph{a product of primes in $I$} if every prime factor $p$ of $n$ lies in $I$. Note that $1$ is a product of primes in $\emptyset$.

Let $n$ be a positive integer and $A$ an abelian group. Consider the morphism $\times n : A \to A$ given by multiplication by $n$. We say that $A$ is \emph{$n$-divisible} if $\times n$ is an epimorphism, and \emph{uniquely $n$-divisible} if it is an isomorphism. We say that $A$ is \emph{(uniquely) $I$-divisible} if it is (uniquely) $n$-divisible for all $n$ that are products of primes in $I$. 

If $\times n$ is a monomorphism, then $A$ is \emph{$n$-torsion-free}, and an abelian group that is $n$-torsion-free for all $n$ divisible by primes in $I$ is said to be \emph{$I$-torsion-free}. At the other extreme, if $\times n: A \to A$ is the $0$-map, we say that $A$ is \emph{$n$-torsion}. If $A$ is $n$-torsion for some product $n$ of primes in $I$, then $A$ is said to be of \emph{$I$-bounded torsion}.

The notation above concerns the group as a whole. Given an element $a \in A$, we say that $a$ is \emph{$I$-torsion} if it is $n$-torsion for some integer $n$ that is a product of primes in $I$; it is \emph{$n$-divisible} (in $A$) if the equation $nx=a$ can be solved for $x$ in $A$, and \emph{uniquely $n$-divisible} if the equation has a unique solution;  $a$ is \emph{(uniquely) $I$-divisible} if it is (uniquely) $n$-divisible for all integers $n$ that are products of primes in $I$.

In all cases above, if $I =\sh P$, then $I$ may be omitted from the notation, so we will write ``uniquely divisible'', ``bounded torsion'', ``torsion-free'' and so on.

Suppose $F: \Ab^\op \to \Ab$ is a contravariant additive functor. From the previous definitions, it is immediate that if $A$ is uniquely $I$-divisible, then $F(A)$ is uniquely $I$-divisible.

Write
\begin{equation*}
  I'=\sh P \sm I.
\end{equation*}
We use $\Z[I^{-1}]$ to denote the subring of $\Q$ consisting of numbers that may be written as $a/b$ where $a$ is an integer and $b$ is a product of primes in $I$. Note that $\Z[\emptyset^{-1}] =\Z$ and $\Z[\sh P^{-1}]=\Q$. We define $\ZIinf$ as the quotient group in the short exact sequence:
  \begin{equation}
    \label{eq:4}
    0 \to \Z \to \Z[I^{-1}] \to \ZIinf \to 0.
  \end{equation}
The group $\ZIinf$ is torsion, and the torsion order of each element is a product of primes in $I$. In the case where $I=\{p\}$ is a singleton, we generally write $\Z[p^{-1}]$ and $\Zpinf$. If $I = \sh P$ is the set of all prime numbers, then we will write  $\Q/\Z$ for $\Z/[\sh P^\infty]$.

\begin{lemma}\label{lem:wasBulletPoint}
 The group $\Z[I^{-1}]$ is uniquely $I$-divisible. The group $\ZIinf$ is $I$-divisible, and uniquely $I'$-divisible
\end{lemma}
\begin{proof}
  The first claim, unique $I$-divisibility of $\Z[I^{-1}]$, is elementary.

  The $I$-divisibility of $\ZIinf$ follows from the same fact for $\Z[I^{-1}]$. Our last claim, about $I'$-divisibility of $\ZIinf$, is vacuously true if $I = \sh P$ and is trivially true if $I' = \sh P$. We assume that $\emptyset \neq I \neq \sh P$.

  Suppose $q$ is a prime not in $I$. For any $x \in \ZIinf$, we can find some product $n$ of primes in $I$ such that $nx=0$. Then $q$ divides $n^{q-1}-1$, so that $\frac{1-n^{q-1}}{q}$ is an integer. If we multiply $\frac{1-n^{q-1}}{q}x$ by $q$, we obtain $x$, so that $\times q$ is indeed surjective. Injectivity of $\times q$ follows from this: for any $y \in \Z[I^{-1}]$, the product $qy$ is an integer if and only if $y$ is.
\end{proof}

Since $\Hom$, $\Ext$ are additive functors, we immediately deduce the following.
\begin{corollary}\label{cor:wereBulletPoints}
For any abelian group $A$:
\begin{enumerate}
\item\label{i:div1} $\Hom(\Z[I^{-1}], A)$ and $\Ext(\Z[I^{-1}], A)$ are uniquely $I$-divisible.
\item\label{i:div2} $\Hom(\ZIinf, A)$ and $\Ext(\ZIinf, A)$ are uniquely $I'$-divisible.
\item\label{i:div3} $\Hom(\ZIinf, A)$ is torsion-free.
\end{enumerate}
\end{corollary}
\begin{proof}
  Points \ref{i:div1} and \ref{i:div2} hold since $\Hom(-, A)$ and $\Ext(-,A)$ are additive and the groups $\Z[I^{-1}]$ and $\ZIinf$ are uniquely $I$- and $I'$-divisible respectively.

  Since uniquely $I'$-divisible implies $I'$-torsion free, to establish \ref{i:div3} we need only show that $\Hom(\ZIinf, A)$ is $I$-torsion-free. The left-exact functor $\Hom(-, A)$ converts the epimorphism $\Z[I^{-1}] \to \ZIinf$ to an injection, so that $\Hom(\ZIinf, A)$ is isomorphic to a subgroup of an $I$-torsion-free group, establishing the result.
\end{proof}

\begin{definition}[{\protect \cite{Bousfield1972}}]
  The \emph{$\Ext$-completion} of an abelian group $A$ at the set of primes $I$ is the group $\Ext(\ZIinf, A)$. 
\end{definition}
We have observed above that the $\Ext$-completion consists of nontorsion or elements whose torsion order is a product of primes in $I$. From the short exact sequence \eqref{eq:4} we deduce that there is a natural homomorphism
\[ A=\Hom(\Z,A) \to \Ext(\ZIinf, A), \]
and we will refer to this as the ``$I$-completion map'' in the sequel, or the ``$p$-completion map'' when $I=\{p\}$. As usual, if $I =\sh P$, the symbol ``$I$'' may be omitted.

If $I_0\subseteq I$ is a subset of $I$, then there is a canonical homomorphism $\Z/(I_0^\infty) \to \Z/(I^\infty)$. In particular, for any set of prime numbers, we obtain a canonical homomorphism
\begin{equation}
  \label{eq:5}
  \bigoplus_{p \in I} \Zpinf \to \ZIinf.
\end{equation}

\begin{proposition} \label{pr:canIsIso}
  The canonical homomorphism of \eqref{eq:5} is an isomorphism.
\end{proposition}
\begin{proof}
  Surjectivity of the homomorphism follows from surjectivity of $\bigoplus_{p \in I} \Z[p^{-1}] \to \Z[I^{-1}]$, where it follows from the existence of partial fraction decompositions for rational numbers.

  For injectivity, suppose that some sequence
  \[ \left( \frac{a_1}{p_1^{s_1}}, \frac{a_2}{p_2^{s_2}}, \dots , \frac{a_r}{p_r^{s_r}} \right) \in \bigoplus_{p \in I} \frac{\Z}{(p^\infty)} \]
  maps to $0$ under the canonical homomorphism, where $0 \le a_i < p_i^{s_i}$ for all $i$. That is, suppose
  \[ \sum_{i = 1}^r \frac{a_i}{p_i^{s_i}} \text{\qquad is an integer.} \]
  By clearing denominators and considering the power of each prime dividing the numerator, we deduce that each $a_i$ must be $0$. This establishes injectivity.
\end{proof}

We have defined the $I$-completion for our own convenience. The next corollary shows that it is sufficient only to define $p$-completions. The proof is that $\Ext$ converts sums in the first variable into products.
\begin{corollary}\label{cor:IcompletionIsProduct}
  Let $A$ be an abelian group and $I$ a set of prime numbers. Then the canonical isomorphism $\bigoplus_{p \in I} \Zpinf \to \ZIinf$ induces an isomorphism
  \[ \Ext(\ZIinf, A) \to \prod_{p \in I} \Ext(\Zpinf, A). \]
\end{corollary}

\subsection{Surjectivity of completion}
\label{sec:surj-compl}

\begin{proposition} \label{pr:surjectivityOfCompletion}
  Let $A$ be an abelian group and $I$ a set of prime numbers. Suppose $\Ext(\ZIinf, A)$ is a torsion group. Then the $I$-completion $A \to \Ext(\ZIinf, A)$ is surjective.
\end{proposition}
\begin{proof}
From the exact sequence \eqref{eq:4}, we derive a long exact sequence
    \begin{equation}
      \label{eq:6}
     \cdots \to \Hom(\Z[I^{-1}], A) \to A \to \Ext(\ZIinf, A) \overset{j}\to \Ext(\Z[I^{-1}], A) \to \Ext(\Z, A) = 0.
    \end{equation}
    It suffices to show that the map $j$ is $0$. The target is uniquely $I$-divisible, so that $\times n^{-1}$ is a well defined automorphism of $\Ext(\Z[I^{-1}], A)$ whenever $n$ is a product of primes in $I$. By hypothesis the source consists of torsion elements, and the orders of these elements are products of primes in $I$. For any $a \in \Ext(\ZIinf, A)$, we can find some $n$ for which $na=0$. Then $0 = nj(a)$ by linearity, whereupon $j(a)=0$.
\end{proof}

In the case where $\Ext(\ZIinf, A)$ is a group of bounded torsion, for instance when it is finite, we can say more. The argument involves a step that we set aside as a lemma for later use.

\begin{lemma} \label{lem:divTorExt}
  Let $I$ be a set of prime numbers and $n$ a product of primes in $I$. Let
  \[ 0 \to A \to B \to C \to 0 \]
  be a short exact sequence of abelian groups in which $A$ is uniquely $I$-divisible and $C$ is $n$-torsion. Then the short exact sequence is split.
\end{lemma}
\begin{proof}
 Since $A$ is uniquely $I$-divisible, multiplication by $n$ is an automorphism of $\Ext(C,A)$. On the other hand, since $C$ is $n$-torsion, multiplication by $n$ on $\Ext(C,A)$ is the $0$-map. We conclude that $\Ext(C,A) =0$. The Yoneda interpretation of $\Ext$ implies the sequence splits.
\end{proof}

\begin{proposition}\label{pr:splitSurjectivityOfCompletionNew}
  Let $A$ be an abelian group and $I$ a set of prime numbers. Suppose $\Ext(\ZIinf, A)$ is a group of $I$-bounded torsion and $\Hom(\ZIinf, A)=0$. There is a short exact sequence
    \begin{equation}
      \label{eq:8}
      \begin{tikzcd}
        0 \rar & \Hom(\Z[I^{-1}], A) \rar{g} & A \arrow[r,"j"'] & \Ext(\ZIinf, A) \arrow[l, bend right, "s"',start anchor = north west] \rar & 0 
      \end{tikzcd}
    \end{equation}
  and a splitting map $s$. In this diagram, $j$ is the $I$-completion map, the image of $g$ is exactly the subgroup of $I$-divisible elements in $A$, and the image of $s$ is the $I$-torsion subgroup of $A$.
\end{proposition}
\begin{proof}
  Since $\Hom(\ZIinf,A)=0$, the long exact sequence \eqref{eq:6} contains the short exact sequence \eqref{eq:8}, using Proposition \ref{pr:surjectivityOfCompletion} for exactness on the right. The group $\Hom(\Z[I^{-1}], A)$ is uniquely $I$-divisible, whereas $\Ext(\ZIinf,A)$ is of $I$-bounded torsion by hypothesis. Lemma \ref{lem:divTorExt} tells us that  a splitting map $s$ can be constructed.
    
  Let $n$ be a product of primes in $I$ with the property that $\times n$ annihilates $\Ext(\ZIinf, A) \iso \im(s)$, and suppose $a$ is $I$-divisible in $A$. We can find some $b$ for which $nb=a$, and we may decompose $b$ uniquely as $b=g(x) + s(y)$. Then $a=nb=g(nx) + s(ny) = g(nx) \in \im(g)$. Therefore $\im(g)$ contains all the $I$-divisible elements of $A$. It is itself $I$-divisible, so it is precisely the subgroup of $I$-divisible elements in $A$.
  
  Finally, since $g$ is injective, and $\Hom(\Z[I^{-1}], A)$ is $I$-torsionfree, we see that every $I$-torsion element of $A$ must lie in $\im(s)$.
\end{proof}

\section{The sphere spectrum over \texorpdfstring{$k$}{k}}
\label{sec:sphere-spectrum-over}

We continue to work over an algebraically closed field $k$ of characteristic $0$. The previous results on completions apply well to the motivic sphere spectrum, $\one$. If $p$ is a prime number and $X$ is a motivic spectrum, then \cite[\S3]{Rondigs2008a} constructs the $p$-completion $X \to X^\wedge_p$, and assures us that there are exact sequences
\begin{equation}
  \label{eq:2}
  0 \to \Ext\Big(\Zpinf, \pi_{r,s}(X)\Big) \overset{i}\to \pi_{r,s}(X^\wedge_p) \overset{j}\to \Hom(\Zpinf, \pi_{r-1,s}(X)) \to 0.
\end{equation}
 Relying on \cite{Rondigs2008a}, we know that the groups $\pi_{r,s}(X^\wedge_p)$ are independent of the particular choice of $k$. These groups are intensively studied, and the following result is well known to the experts in the field. We learned the outline of the proof from Dan Isaksen and Zhouli Xu.
\begin{proposition} \label{pr:AdamsConvergeFinite}
  Let $s$ be an integer different from $0$ and let $w$ be an integer. Let $p$ be a prime number. The group $\pi_{s,w}(\one^{\wedge}_p)$ is a finite abelian $p$-group, and it vanishes if $s< 2p-3$.
\end{proposition}
\begin{proof}
  For any given weight $w$, there is an Adams--Novikov spectral sequence converging to $\pi_{*,w}(\one^{\wedge}_p)$, constructed in \cite{Hu2011b} for $p=2$ and in \cite[Thm.~3.1]{Ormsby2014} for odd primes. 
 It is a convergent spectral sequence of trigraded groups
  \[ \Eoh_2^{n,t,w} = \Ext^{n,t,w}_{\BP_{*,*}\BP}(\pi_{*,*}\BP, \pi_{*,*}\BP) \Longrightarrow \pi_{t-n,w}(\one^\wedge_p), \]
  where $\BP$ is the $p$-completion of the $p$-local motivic Brown--Peterson spectrum, which was first constructed in \cite{Vezzosi2001} and \cite{Hu2001b}.
  The differentials satisfy $d_r : \Eoh_r^{n,t,w}\to \Eoh_r^{n+r, t+r-1,w}$, and in particular, the weight-grading $w$ is preserved by the differentials so that we may view this as a family of spectral sequences, one for each $w \in \Z$. We fix $w$ for the rest of this proof. 

  Comparison of the motivic $\Eoh_2$-page with the $\Eoh_2$-page of the classical Adams--Novikov spectral sequence is given by \cite[Equation (36)]{Hu2011b} in the case of $p=2$, and by \cite[Prop.~2.4.4(2)]{Stahn2018} when $p$ is odd. The comparison is given by complex realization, and on the level of groups yields
  \begin{equation}
    \label{eq:1}
    \Ext^{n,t,w}_{\BP_{*,*}\BP}(\pi_{*,*}\BP, \pi_{*,*}\BP) \iso
    \begin{cases}
      \Ext^{n,t}_{\BPtop_*\BPtop}(\pi_*\BPtop, \pi_*\BPtop) &\text{if $w \le t/2$;} \\
      0 & \text{otherwise.}
    \end{cases}
  \end{equation}
  Here the notation $\BPtop$ is used for the $p$-complete classical Brown--Peterson spectrum. We will write $\Ext^{n,t,w}$ for these groups for the sake of brevity.
  
  Similarly, write $\Ext^{n,t}$ for $\Ext^{n,t}_{\BPtop_*\BPtop}(\pi_*\BPtop, \pi_*\BPtop)$. A well-known algebraic presentation of these groups, which may be found in \cite[Prop.~3.16]{Ravenel1978}, shows that these groups are finitely generated modules over the ring of $p$-adic integers $\Z^{\wedge}_p$. This presentation also makes it clear that the groups vanish unless $n \ge 0$ and $t \ge 0$. The statement \cite[Thm.~5.2.1(a)]{Ravenel1986} implies that they are abelian $p$-groups, and therefore finite, except when $n=t=0$. Sparseness results  \cite[Cor.~3.1]{Novikov1967} imply that these groups also vanish when $t \not \equiv 0 \pmod{2p-2}$ and when $t<2n(p-1)$. Beware that \cite{Novikov1967} uses ``$Q_p$'' to denote the ring of $p$-adic integers, and argues using $\MU$ rather than $\BP$, but the implication for $\Ext^{n,t}$ is immediate. The results of this paragraph also apply to the subgroups (up to isomorphism) $\Ext^{n,t,w}$ of $\Ext^{n,t}$.

 The vanishing results imply two things: first, there are only finitely many nonzero groups $\Ext^{n,t,w}$ for any given value of $s$, owing to the constraint $0 \le 2(p-1)n \le t$; second, the smallest nonzero value of $t-n$ for which there is a group $\Ext^{n,t}$ that is not $0$ is $t-n=2p-3$, attained by the pair $(n,t) = (1,2p-2)$. This is reflected in the well-known fact that $p$-torsion arises in the classical stable homotopy groups for the first time at $\pi_{2p-3}(\SpS)$.

  The $E_\infty$-page of the motivic Adams--Novikov spectral sequence converging to $\pi_{t-n,w}(\one^\wedge_p)$ is a subquotient of the $E_2$-page, and it follows that when $t-n \neq 0$, the group $\pi_{t-n,w}(\one^\wedge_p)$ has a finite filtration whose associated graded groups are finite abelian $p$-groups. In particular, $\pi_{t-n,w}(\one^\wedge_p)$ is a finite abelian group when $t-n\neq 0$. Additionally, this group vanishes when $ t-n < 2p-3$, except when $t-n=0$.
\end{proof}

Let us write $\MZ$ for the motivic cohomology spectrum. There exists a unit map $\one \to \MZ$, arising as a unit map of an adjunction between the stable homotopy category and a triangulated category of motives as in \cite[\S 5.3.35]{Cisinski2019}. This induces a unit homomorphism $\eta: \pi_{s,w}(\one) \to \Hoh^{-s}(\Spec k; \Z(-w))$. 

We are now in a position to prove our first main theorem, which we restate for the convenience of the reader.
\surj*
\begin{proof}
  When $s\neq 0$, the group $\prod_{p \in \sh P}\pi_{s,w}(\one^\wedge_p)$ is finite by Proposition \ref{pr:AdamsConvergeFinite}. In particular, it is $n$-torsion for some positive integer $n$.

  For a fixed $p$, in the exact sequence \eqref{eq:2}, the group $\Hom(\Zpinf, \pi_{s-1,w}(\one))$ is a  torsion-free quotient of a torsion group. It vanishes, yielding an isomorphism $\Ext(\Zpinf, \pi_{s,w}(\one)) \iso \pi_{s,w}(\one^\wedge_p)$. From this, Corollary \ref{cor:IcompletionIsProduct} tells us that the comparison
  \[ \Ext(\Q/\Z, \pi_{s,w}(\one)) \isomto \prod_{p \in \sh P} \pi_{s,w}(\one^\wedge_p)  \]
  is an isomorphism.

  Since this group is finite, Proposition \ref{pr:surjectivityOfCompletion} applies, saying that the completion map
 \[ \pi_{s,w}(\one) \to \Ext(\Q/\Z, \pi_{s,w}(\one))\] 
  is surjective. When $s \neq -1$, we have also seen that $\Hom(\Zpinf, \pi_{s+1,w}(\one)) = 0$, so that Proposition \ref{pr:splitSurjectivityOfCompletionNew} applies, telling us the completion is split surjective and the kernel is the subgroup of divisible elements. Furthermore, the kernel is isomorphic to $\Hom(\Q, \pi_{s,w}(\one))$, so is a $\Q$-vector space in its own right.

  In the case of $s=-1$, a different argument is required. In this case, Proposition \ref{pr:AdamsConvergeFinite} tells us that that $\pi_{-1,w}(\one^\wedge_p)=0$ for all primes $p$. Therefore the completion map is split surjective for trivial reasons. We also see that $\Ext(\Zpinf, \pi_{-1,w}(\one)) = 0$, by short exact sequence \eqref{eq:2}. We may, however, have a nontrivial presentation arising from exact sequence \eqref{eq:6}:
  \[ \Hom(\Q/\Z, \pi_{-1,w}(\one)) \to \Hom(\Q, \pi_{-1,w}(\one)) \to \pi_{-1,w}(\one) \to 0. \]
  In this case we may conclude that $\pi_{-1,w}(\one)$ is a quotient of a uniquely divisible group, and so is divisible.
  
  \smallskip
  
  For the rest of this proof, we assume $s \neq -1$. There is a diagram
  \begin{equation}
    \label{eq:9}
    \begin{tikzcd}
      0 \rar & K_{s,w} \rar{i} &  \pi_{s,w}(\one) \rar \dar{\eta} &  \prod_{p \in \sh P} \pi_{s,w}(\one^\wedge_p) \rar &  0 \\
      & & \Hoh^{-s}(\Spec k ; \Z(-w)))
    \end{tikzcd}
  \end{equation}
  where the row is an exact sequence, so that $i : K_{s,w} \to \pi_{s,w}(\one)$ denotes inclusion of the kernel of the completion map. It remains to prove that $\eta \circ i$ is an isomorphism.

  The group $K_{s,w}$ was observed to be a $\Q$-vector space, and the cokernel of $i$ is finite, so that applying $\Q \otimes_\Z -$ to \eqref{eq:9} yields a diagram
 \begin{equation*}
    \begin{tikzcd}
      0 \rar & K_{s,w} \arrow[r,"i","\iso"'] &  \Q \tensor_\Z \pi_{s,w}(\one) \rar \arrow[d, "\iso"', "\id_\Q \tensor \eta"] &  0  \\
      & & \Hoh^{-s}(\Spec k ; \Q(-w)))
    \end{tikzcd}
  \end{equation*}
  in which the row is exact.
  
  A proof that 
  \begin{equation*}
    \id_\Q \tensor_\Z \eta \colon \Q \tensor_\Z \pi_{s,w}(\one) \to \Hoh^{-s}(\Spec k; \Q(-w))
  \end{equation*}
  is an isomorphism is outlined in \cite[Thm.~6.2]{Ormsby2014}, following work of Morel (unpublished) and  \cite[\S\S 5.3, 16.2]{Cisinski2019}. Strictly speaking, these references assert only that there is an abstract isomorphism, but from \cite[\S~5.3.35]{Cisinski2019} one sees that this abstract isomorphism is induced by the unit map. In particular $\id_\Q \tensor (\eta \circ i)$ is an isomorphism.

  We claim that the motivic cohomology groups $\Hoh^{-s}(\Spec k; \Z(-w))$ are uniquely divisible, i.e., admit the structure of $\Q$-vector spaces, except when $w \le 0$ and $s=0,-1$. This is a consequence of the long exact sequence
  \[ \cdots \to \Hoh^{-s}(\Spec k; \Z(-w)) \overset{\times n}{\to} \Hoh^{-s}(\Spec k; \Z(-w)) \to \Hoh^{-s}(\Spec k ; \Z/(n)(-w)) \to \Hoh^{-s+1}(\Spec k; \Z(-w)) \to \cdots \]
  and the fact that
  \begin{equation}
    \label{eq:7}
    \Hoh^{-s}(\Spec k; \Z/(n)(-w)) \iso
    \begin{cases}
      \Z/(n) &\text{  if $s=0$ and $w \le 0$;} \\
      0 & \text{ otherwise}.
    \end{cases}
  \end{equation}
  The isomorphisms in \eqref{eq:7} are justified by:
  \begin{itemize}
  \item The vanishing of motivic cohomology for nonpositive weights, i.e., when $w \ge 0$ \cite[Cor.~4.2]{Mazza2006};
  \item Dimensional vanishing, i.e., when $s<w$ \cite[Thm.~3.5]{Mazza2006};
  \item The Beilinson–Lichtenbaum version of the norm-residue isomorphism, \cite[Thm.~1.8(b)]{Haesemeyer2019}, which says that 
  \begin{equation*}
    \Hoh^{-s}(\Spec k; \Z/(n)(-w))\iso \Hoh^{-s}_\et(\Spec k; \mu_n^{\tensor(-w)})
  \end{equation*}
   when $w \le \min\{s,-1\}$.  Since $k$ is algebraically closed, the étale cohomology group $\Hoh^{-s}_\et(\Spec k; \mu_n^{\tensor(-w)})$ is isomorphic to $\Z/(n)$ when $-s=0$ and is $0$ otherwise.
  \end{itemize}
  This proves our claim.

   Therefore the map $\eta \circ i : K_{s,w} \to \Hoh^{-s}(\Spec k; \Z(-w))$ is a homomorphism between $\Q$ vector spaces that yields an isomorphism after application of $\Q \tensor_\Z -$. In particular, it is itself an isomorphism.   
\end{proof}

\begin{corollary} \label{cor:SphereIso}
   Let $s,w$ be integers satisfying $s \neq 0$, $w \ge -1$ and $(s,w) \neq (-1,-1)$. The comparison map
  \[ \pi_{s,w}(\one) \to \prod_{p \in \sh P} \pi_{s,w}(\one^\wedge_p) \]
  is an isomorphism.
\end{corollary}
\begin{proof}
  First, suppose $s \neq -1$ and $w \ge -1$. Then Theorem \ref{th:surjForSphere} tells us the map is surjective and the kernel is the motivic cohomology group $\Hoh^{-s}(\Spec k; \Z(-w))$. This group vanishes by \cite[Ch.\ 4]{Mazza2006}.

  Now suppose $s=-1$ and $w \ge 0$. Here $\pi_{-1,w}(\one) = \pi_{-1-w + w \alpha}(\one) = 0$ by Morel's calculations \cite[Cor.\ 6.43]{Morel2012}, so that surjectivity of the comparison map implies isomorphism.
\end{proof}

Recall that we write $\SpS$ for the sphere spectrum in classical topology. For any $s,w$, there are complex realization maps
\[ \pi_{s,w}(\one) \to \pi_s(\SpS), \]
by virtue of the realization functors \cite[Thm.~A.45]{Panin2009a}. Recall from \cite{Rondigs2008a} that the $p$-completion $X^{\wedge}_p$ is constructed by means of Bousfield localization with respect to the motivic Moore spectrum. One construction of this spectrum is $\one/(p) = \SpS/(p) \wedge \one$, where $\SpS/(p)$ denotes the classical Moore spectrum. Complex realization of $\one/(p)$ yields $\SpS/(p)$, since this is a monoidal functor the realization of $\one$ is $\SpS$, the monoidal unit. The theory of Bousfield localization, e.g., \cite[Thm.~3.3.20]{Hirschhorn2003}, implies that the two functors, $p$-completion and complex realization, commute up to equivalence. There are commutative squares
\[
  \begin{tikzcd}
    \pi_{s,w}(\one) \rar \dar &  \pi_s(\SpS) \dar  \\ \pi_{s,w}(\one^{\wedge}_p) \rar & \pi_s(\SpS^{\wedge}_p).
  \end{tikzcd}
\]

\begin{corollary}\label{cor:isoForSphereRealization}
  Suppose $k$ is a subfield of $\C$. Let $s,w$ be integers satisfying $s \neq 0$ and $w \le \frac{1}{2}s+1$. Then the complex realization
  \[ \pi_{s,w}(\one) \to \pi_s(\SpS) \] is split surjective and the kernel is the subgroup of
   divisible elements. If additionally $s \neq -1$, the kernel is isomorphic to the motivic cohomology group $\Hoh^{-s}(\Spec k; \Z(-w))$.
\end{corollary}
\begin{proof}
  For all $s$ and $w$, the following diagram commutes
  \[
    \begin{tikzcd}
      \pi_{s,w}(\one) \rar \dar & \pi_s(\SpS) \dar \\ \prod_{p \in \sh P} \pi_{s,w}(\one^\wedge_p) \rar{} & \prod_{p \in \sh P} \pi_s(\SpS^\wedge_p),
    \end{tikzcd}
  \]
  where the horizontal arrows are given by complex realization and the vertical arrows are completions. When $s \neq 0 $, connectivity of the sphere spectrum and Serre's thesis imply that $\pi_s(\SpS)$ is finite so that the right vertical arrow is an isomorphism. When $w \le \frac{1}{2}s + 1$, \cite{Gheorghe2017} (for the prime $2$) and \cite{Stahn2018} (for odd primes) tell us that the lower horizontal map is an isomorphism.

  Therefore, when $s \neq 0$ and $w \le \frac{1}{2}s+1$, whatever properties enjoyed by the completion map $\pi_{s,w}(\one) \to \prod_{p \in \sh P} \pi_{s,w}(\one^\wedge_p)$ are also enjoyed by the complex realization map. The result now follows from Theorem \ref{th:surjForSphere}. 
\end{proof}

Note that Corollary \ref{cor:isoForSphereRealization} assures us that the complex realization is an isomorphism in many cases, since $\Hoh^{-s}(\Spec k; \Z(-w))=0$ when $w \ge -1$ and $s \neq 0, -1$.

\begin{corollary}\label{cor:realIsoQbar}
  Suppose the Beilinson–Soulé conjecture holds for $k$. Suppose $s$ is a positive integer and $w$ is an integer satisfying $w \le \frac{1}{2}s+1$. The realization map
  \[ \pi_{s,w}(\one) \to \pi_{s}(\SpS) \]
  is an isomorphism. In particular, this map is an isomorphism when $k$ is the field $\bar \Q$ of algebraic numbers.
\end{corollary}
\begin{proof}
 The Beilinson--Soulé conjecture asserts that $\Hoh^{-s}(\Spec k ;\Z(-w)) = 0$ when $s \ge 1$, independently of $w$. Assuming this, Corollary \ref{cor:isoForSphereRealization} immediately yields the asserted isomorphism.
  
 This Beilinson–Soulé conjecture is known to hold over number fields by and results of Bloch, explained in \cite[\S~1.6]{Deligne2005}. By \cite[Lem.\ 3.8]{Mazza2006}, the conjecture therefore holds for $k = \bar \Q$. 
\end{proof}

\subsection{Further remarks on the homotopy groups of the sphere spectrum}
\label{sec:furth-remarks-homot}

In Figure \ref{fig:chartStable}, inspired by \cite[Fig.~1]{Gheorghe2017}, we attempt a visual presentation of some of what we know about the groups $\pi_{s,w}(\one)$. The $(s,w)$-plane is presented, divided into a number of regions by lines and rays. Light-grey dots indicate a group $\pi_{s,w}(\one)$ that is not known to vanish for dimensional reasons. Vanishing of the groups above the line $s=w$ was established by Morel \cite[Thm.~6.1.8]{Morel2005}. We proceed around the rest of the diagram clockwise, starting from the top right corner.

The region marked ``finite'' is of great interest. The groups $\pi_{s,w}(\one)$ appearing here are finite abelian groups for which the comparison $\pi_{r,s}(\one) \to \pi_r(\SpS)$ is not in general an isomorphism. To $2$-primary eyes, this region is further divided in \cite[Fig.~1]{Gheorghe2017} into a cone above $w=\frac{3}{5}s+1$ consisting of $\eta$-periodic classes and a cone below this line which they call ``not understood.'' Odd-primary behaviour in this region is less mysterious, but still complicated and interesting, being governed by the differentials in the classical Adams--Novikov spectral sequence according to \cite[Prop.~2.4.8]{Stahn2018}.

In the region we have marked $\pi_s(\SpS)$, on or below the line $w=\frac{1}{2}s+1$, strictly to the left of $s=0$, and on or above the line $w=-1$, the comparison map $\pi_{s,w}(\one) \to \pi_s(\SpS)$ is an isomorphism. In \cite{Levine2014}, Levine proved that the comparison along the entire $s$-axis is an isomorphism. Our proofs ultimately rely on his result, since the isomorphisms $\pi_{s,w}(\one^{\wedge}_p) \to \pi_s(\SpS^\wedge_p)$ below the line $w = \frac{1}{2}s+1$ are established by reference to this, combined with an observation that there are no nonzero $\tau$-torsion classes in this region.

Below the line $w=-1$, there is a quadrant labelled ``B--S'' for ``Beilinson--Soulé''. In this region, the comparison $\pi_{s,w}(\one) \to \pi_s(\SpS)$ is surjective, and has $\Hoh^{-s}(\Spec k; \Z(-w))$ as a kernel. Since $s > 0$, the Beilinson--Soulé conjecture is that this kernel vanishes.

Finally, in the third quadrant, there is a region bounded by $s=w$ (inclusive) and $s=-1$ (exclusive) where there is an isomorphism $\pi_{s,w}(\one) \isomto \Hoh^{-s}(\Spec k; \Z(-w))$. For what happens when $s=-1$, see Remark \ref{rem:conjectureStem-1}.

We have lightly shaded the three diagonal lines of coweight $0$, $1$ and $2$, where superb results have been obtained. We do not interact with these results much in this paper, but we give a brief account here.

The groups of coweight $0$ have been determined by Morel, even unstably as in \cite[Cor.~6.43]{Morel2012}. A detailed analysis of the Adams and the Adams--Novikov-spectral sequences was carried out when $c=1$ in \cite{Ormsby2014}, allowing the calculation of the sheaf structure on $\bpi_{1+n\alpha}(\one)$ over ``low-dimensional fields'' (see \cite{Ormsby2014} for a precise definition), not just the structure of the global sections over algebraically closed fields. This was superseded, however, by slice spectral sequence techniques in \cite{Rondigs2019}, \cite{Rondigs2024}. For coweights $c=0,1,2$, these papers give calculations of the sheaf $\bpi_{c + w\alpha}(\one)$ over an arbitrary field. We remark, however, that all the calculations mentioned in this paragraph are determinations of homotopy sheaves in terms of the motivic cohomology of the ground field, so that without further work, they do not establish any cases of the Beilinson--Soulé conjecture.

\begin{figure}
\begin{tikzpicture}[scale=0.5] 
\centering

\draw[rounded corners=4pt, draw=white, fill=black!04!white
] (-9.3,-9.3) -- (-6.7,-9.3) -- (12.3,9.7) -- (12.3,12.3) -- (11.7, 12.3) -- (-9.3, -8.7) --cycle;


\foreach \i in {-9,...,12}
 \foreach \j in {-9,...,\i}
   \draw (\i,\j)[color=black!30, fill=black!30] circle(2pt);

\draw [ thick, ->, mygray] (0,-9) -- (0,12);
\node [left= 2pt, fill=white, rounded corners=4pt, inner sep=2pt] at (0,10) {\Large $w$ };
\draw [ thick, ->, mygray] (-9,0) -- (12,0);
\node [above=3pt] at (10.5,0) {\Large $s$ };

\draw [thick] (2,2) -- (12,7);
\node [rotate=26.5] at (7.1,4) { $w= \frac12s+1$ };

\draw [thick] (0,-9) -- (0,0);
\node [right, fill=white, rounded corners=4pt, inner sep=1pt,right, rotate=0] at (0.1,-4) { $s=0$ };

\draw [thick, dashed] (-1,-9) -- (-1,-1);

\draw [thick] (0,-1) -- (12,-1);
\node [below] at (4,-1) {  $w=-1$ };

\draw [thick] (-9,-9) -- (12, 12);
\node [fill=white, rounded corners=4pt, inner sep=1pt, right, rotate=45] at (3.2, 4) {coweight $0$, \quad $s=w$};

\node [fill=white, rounded corners=4pt, inner sep=2pt] at (-4,2) {\Huge $0$};
\node [fill=white, rounded corners=4pt, inner sep=2pt] at (10,8) {\LARGE finite};
\node [fill=white, rounded corners=4pt, inner sep=2pt] at (7,2) {\LARGE $\pi_{s}(\SpS)$};
\node [fill=white, rounded corners=4pt, inner sep=2pt] at (7,-4) {\LARGE B--S};
\node [fill=white, rounded corners=4pt, inner sep=2pt] at (-4,-7.5) {$\Hoh^{-s}(\Spec k; \Z(-w))$};
\end{tikzpicture}

\caption{The groups $\pi_{s,w}(\one)$ over $k=\bar k$, with special reference to their comparison with $\pi_s(\SpS)$ and $\Hoh^{-s}(\Spec k; \ZZ(-w))$.}
\label{fig:chartStable}
\end{figure}

\begin{remark} \label{rem:conjectureStem-1}
  We prefer integral cohomology $\Hoh^{-s}(\Spec k; \Z(-w))$ to the rational cohomology $\Hoh^{-s}(\Spec k; \Q(-w))$ in Theorem \ref{th:surjForSphere} because the theorem holds in full even when $(s,w)=(-1,-1)$ and $(s,w) = (-1,-2)$, provided integer cohomology is used. The classical homotopy group in question is $\pi_{-1}(\SpS) = 0$, and the calculations due to \cite[Cor.\ 6.43]{Morel2012} and \cite[Table 1]{Rondigs2024} give us
  \[ \pi_{-1,-1}(\one) \iso \K^{\MW}_1(k) \iso \K^\M_1(k) \iso k^\times \iso \Hoh^1(\Spec k; \Z(1)) \]
  and
  \[ \pi_{-1,-2}(\one) \iso \Hoh^1(\Spec k; \Z(2)), \]
  both of which are divisible groups.
  Note that in \cite[Table 1]{Rondigs2024}, all the other terms vanish over an algebraically closed field of characteristic $0$, so that their spectral sequence, the slice spectral sequence for the sphere spectrum, collapses at the $\Eoh_2$-page.
\end{remark}

\begin{conjecture}
  We conjecture that over an algebraically closed field $k$, the completion map $\one \to \bigvee_{p \in \sh P} \one^{\wedge}_p$ and the unit map $\eta: \one\to \MZ$ induce isomorphisms
  \[ \pi_{s,w}(\one)\to \Hoh^{-s}(\Spec k; \Z(-w))\oplus \prod_{p \in \sh P} \pi_{s,w}(\one^{\wedge}_p) \]
for all $s \neq 0$. This is the content of Theorem \ref{th:surjForSphere} when $s\neq -1$, so only the case of $s=-1$ remains unproved.
\end{conjecture}

\begin{remark}
  If $\bar k$ is not a subfield of $\C$, but is an algebraically closed field of characteristic $0$, a result similar to Corollary \ref{cor:isoForSphereRealization} holds, by virtue of \cite{Rondigs2008a}. The realization map may be replaced by a zigzag $\pi_{s,w}(\one) \leftarrow \pi_{s,w}(\one_{\bar \Q}) \to \pi_s(\SpS)$, where $\one_{\bar \Q}$ denotes the motivic sphere spectrum over $\bar \Q$. This zigzag consists of isomorphisms when $0 \le w \le \frac{1}{2}s+1$. When $w < 0$ and $s \not \in \{-1,0\}$, we can say that $\pi_{s,w}(\one)$ admits a decomposition as a direct sum of a finite group isomorphic to $\pi_s(\SpS)$ and a uniquely divisible group $\Hoh^{-s}(\Spec k; \Z(-w))$.
\end{remark}

\begin{remark}\label{rem:realizationBoundsInCoweight}
  We write the bounds of Corollary \ref{cor:isoForSphereRealization} in coweight-weight terms for future reference.
  Let $c,w$ be integers satisfying $c+w \neq 0$ and $w \le c+ 2$. Then the realization map
  \[ \pi_{c+w\alpha}(\one) \to \pi_{c+w}(\SpS) \]
  is split surjective and the kernel is the subgroup of divisible elements. Except when $c+w=-1$, this kernel is isomorphic to the motivic cohomology group $\Hoh^{-c-w}(\Spec k; \Z(-w))$. In particular, the realization map is an isomorphism when $w \ge -1$, and if the Beilinson--Soulé vanishing conjecture holds for $k$, it is an isomorphism when $c+w > 0$.
\end{remark}

\section{Unstable homotopy groups of spheres and  Stiefel varieties}
\label{sec:unst-homot-groups}

We use the results of \cite{Asok2024} to extend Corollary \ref{cor:isoForSphereRealization} to $\pi_{s,w}(S^{a,b})$, the unstable motivic homotopy groups of spheres. For the sake of simplicity of exposition, we assume in this section that the ground field $k=\bar k$ is embedded in $\C$.

Recall that there are adjoint functors $\Sigma_{\PP^1}^\infty \dashv \Omega^\infty_{\PP^1}$ between the categories of pointed motivic spaces and motivic $\PP^1$-spectra. One writes $Q$ for $\Omega^\infty_{\PP^1}\Sigma^\infty_{\PP^1}$. The unstable homotopy sheaves of $QX$ are isomorphic to the stable homotopy sheaves of $X$. There is a unit natural transformation $\id \to Q$, which we call the \emph{stabilization}.

To pass between stable and unstable motivic homotopy groups, we use the following motivic Freudenthal theorem.
\begin{proposition}[{Asok--Bachmann--Hopkins, \protect  \cite{Asok2024}}]\label{pr:motivicFreudenthalForSpheres}
  Let $a,b$ be nonnegative integers and $s,w$ be integers. Suppose the following inequalities hold:
  \begin{enumerate}
  \item $b \ge 2$ and $a-b \ge 2$;
  \item $s-w \le \min\{a-b-2,b-2\}$;
  \end{enumerate}
  then the stabilization $S^{a,b} \to Q S^{a,b}$ induces an isomorphism of homotopy sheaves
  \[ \bpi_{a+s,b+w}(S^{a,b}) \to \bpi_{s,w}(\one). \]  
\end{proposition}
\begin{proof}
  In the notation of \cite{Asok2024}, $S^{a,b} \in O(S^{a,b})$. Therefore by \cite[Theorem 6.3.4]{Asok2024}, the fibre $\sh F$ of the stabilization $S^{a,b} \to QS^{a,b} = \Omega^\infty_{\P^1} \Sigma^\infty_{\P^1} S^{a,b}$ lies in $O(S^{p,2b})$ where $p$ is the minimum of $2a-1$ and $a+2b-1$. Then by \cite[Lemma 3.1.19]{Asok2024}, the motivic space $\sh F$ is $p-2b-1$-$\Aone$-connected. Recall that $\Aone$-connectivity is measured by coweight, see \cite[p.~165]{Morel2012}. That is to say
    \begin{equation*}
      \bpi_{a+s,b+w}(S^{a,b}) \to \bpi_{s,w}(\one)
    \end{equation*}
  is an isomorphism provided
  \[ a-b+s-w \le \min\{2a-1, a+2b-1\}-2b-1 = \min\{2(a-b)-2, a-2\} \]
  or in other words
  \[ s-w \le \min\{a-b-2, b-2\}, \]
  which is what we wanted
\end{proof}

\begin{remark}\label{rem:motivicFreudenthalForSpheresCoweight}
  The result is slightly easier to read when we use coweight-weight grading rather than stem-weight grading. In those terms, it reads as follows:
  Let $x,y, c,w$ be integers. Suppose the following inequalities hold:
  \begin{enumerate}
  \item $x \ge 2$ and $y \ge 2$;
  \item $c \le \min\{x-2, y-2\}$.
  \end{enumerate}
  Then the stabilization map
  \[ \bpi_{x+c + (y+w)\alpha} (S^{x+y\alpha}) \to \bpi_{c+w\alpha}(\one) \]
  is an isomorphism. 
\end{remark}

The next result is little more than a combination of Corollary \ref{cor:SphereIso} with Proposition \ref{pr:motivicFreudenthalForSpheres}, but we set it aside as a proposition for later reference.

\begin{proposition}\label{prop:comparisonOfUH}
  Let $d,e,x,y$ be nonnegative integers satisfying:
  \begin{enumerate}
  \item\label{ineq:1} $x \ge 2$ and $y \ge 2$;
  \item\label{ineq:2} $d \le \min \{2x-2, x+y-2\}$;
  \item\label{ineq:3} $e-y\le d-x+2$;
  \item\label{ineq:4} $d+e \neq x+y$.
  \end{enumerate}
  Then the classical unstable homotopy group $\pi_{d+e}(S^{x+y})$ lies in the stable range, the realization map
  \[ \pi_{d+e\alpha}(S^{x+ y\alpha}) \to \pi_{d+e}(S^{x+y}) \iso  \pi_{d+e-(x+y)}(\SpS) \]
  is split surjective, and the kernel is the subgroup of arbitrarily divisible elements. Additionally
  \begin{itemize}
  \item if $d+e \neq x+y-1$, then the kernel is a uniquely divisible group.
  \item if $y-1 \le e$ the realization map is an isomorphism. This may be promoted to $\min\{y-1,y+x-d+1\} \le e$ if the Beilinson--Soulé conjecture holds for $k$.
  \end{itemize}
\end{proposition}
\begin{proof}
  We consider a commutative square in which horizontal arrows are stabilization maps and vertical arrows are realization maps:
  \begin{equation*}
    \begin{tikzcd}
      \pi_{d+e\alpha}(S^{x+y\alpha}) \rar{\iso} \dar & \pi_{(d-x) + (e-y)\alpha}(\one) \arrow[d, two heads]\\
      \pi_{d+e}(S^{x+y}) \rar{\iso}  & \pi_{d + e-x-y}(\SpS).
    \end{tikzcd}
  \end{equation*}
  Inequality \ref{ineq:2} implies \textit{a fortiori\/} that
  \[ d \le \frac{(2x-2) + (x+y-2)}{2} = \frac{3x+y-4}{2}. \]
  Combine this with
  \[ d+e \le d + d-x+y+2 = 2d-x+y+2 \]
  to deduce
  \[ d+e \le 2 \frac{3x+y-4}{2} -x +y + 2 =  2x+2y - 2.\]
  Therefore the classical Freudenthal suspension theorem applies to say that the lower arrow in the diagram is an isomorphism.
  
  The top arrow is an isomorphism by Proposition \ref{pr:motivicFreudenthalForSpheres}, using inequalities \ref{ineq:1} and \ref{ineq:2}.

  The right arrow is split surjective and the kernel is the subgroup of divisible elements, by Corollary \ref{cor:isoForSphereRealization} using inequalities \ref{ineq:3} and \ref{ineq:4}. The two additional claims follow from Corollaries \ref{cor:isoForSphereRealization} and \ref{cor:realIsoQbar}.
\end{proof}

To illustrate the result, we provide Figure \ref{fig:chart}. The figure shows what Proposition \ref{prop:comparisonOfUH} tells us about the groups $\pi_{d+e,e}(S^{x+y,y}) = \pi_{d+e\alpha}(S^{x+y\alpha})$, and in particular about the realization map
\[ \pi_{d+e\alpha}(S^{x+y\alpha}) \to \pi_{d+e}(S^{x+y}). \] We have fixed values for $x,y$ (subject to the condition that $x, y \ge 2$), and indicate regions of interest in the $d,e$-plane. Light grey dots indicate the presence of groups $\pi_{d+e\alpha}(\SpS^{x+y\alpha})$ that are not forced to vanish for reasons of connectivity.

The figure may be obtained from Figure \ref{fig:chartStable} by changing stem--weight bigrading to stem--coweight bigrading, moving the origin to account for the passage from stable to unstable homotopy groups, and the drawing of a line $d=\min\{2x-2,x+y-2\}$. The right of this line is where the motivic Freudenthal suspension theorem does not apply, so our methods break down. On or to the left of this line, the theorem applies so that the unstable group $\pi_{d+e\alpha}(S^{x+y\alpha})$ agrees with $\pi_{d-x+(e-y)\alpha}(\one)=\pi_{d+e-x-y, e-y}(\one)$. 

Above the diagonal line $e-y=d-x+2$, the comparison of the relevant motivic and classical stable homotopy groups of spheres is not understood. This is labelled ``finite'' because it corresponds to a portion of the region with the same label in Figure \ref{fig:chartStable}.

When $d+e=x+y$, the stable motivic homotopy group is in the $0$-stem, which is not torsion so our results on completion do not apply. To a lesser extent, this affects the line $d+e=x+y-1$ corresponding to the $-1$-stem, where our understanding of the kernel is worse than elsewhere. To the inequalities of Proposition \ref{prop:comparisonOfUH} another can be added: when $d<x$, the source group $ \pi_{d+e\alpha}(S^{x+ y\alpha})$ (corresponding to $\pi_{d-x+(e-y)\alpha}(\one)$ of negative coweight) vanishes by a connectivity result of \cite{Morel2012}.

In Figure \ref{fig:chart}, we identify regions in which the realization map $\pi_{d+e\alpha}(S^{x+y\alpha}) \to \pi_{d+e}(S^{x+y})$ takes on certain characteristics. These correspond to sub-regions of the similarly named regions of Figure \ref{fig:chartStable} under the comparison between stable and unstable homotopy groups.

To the left of the line $d=x$, the source of the realization map is $0$, and we do not label this region further. Second, marked by vertical hatching, there is a closed pentagonal region labelled  ``$\pi_{d+e}(S^{x+y})$''. This is a region where the map is known to be an isomorphism. Third, there is a triangular region marked with crosshatching and labelled ``B--S''. Here the map is an isomorphism contingent on the Beilinson--Soulé vanishing conjecture. This region contains its rightmost edge but not the rest of the boundary. Fourth, marked by horizontal hatching, there is a vertical strip with slanted top, containing its boundary, and labelled $\Hoh^{x+y-d-e,y-e}$. Here the target of the comparison map vanishes, and the group $\pi_{d+e\alpha}(S^{x+y\alpha})$ is known to be divisible. Except possibly along the $-1$-stem line, it is isomorphic to $\Hoh^{x+y-d-e}(\Spec k; \Z(y-e))$.

\begin{figure}
\begin{tikzpicture}[scale=0.5] 
\centering

\draw [pattern= horizontal lines, pattern color = black!20, draw=white] (8,8)  -- (14,2) -- (14,-2) -- (8,-2);

\draw [pattern= vertical lines, pattern color = black!20] (9,8) --(14,8) -- (14,17) -- (8,11) -- (8,9);

\draw [pattern=crosshatch, pattern color = black!20, draw= white] (9,8) -- (14,8) -- (14,3);


\draw [ thick, ->, mygray] (0,-2) -- (0,20);
\node [left= 4pt, fill=white, rounded corners=4pt, inner sep=2pt] at (0,13) {\Large $e$ };
\draw [ thick, ->, mygray] (-1,0) -- (20,0);
\node [below=4pt] at (19.5,0) {\Large $d$ };
\foreach \i in {8,...,20}
 \foreach \j in {-2,...,20}
   \draw (\i,\j)[color=black!30, fill=black!30] circle(2pt);

\draw [very thick] (14,-2) -- (14,20) ;
\node [fill=white, rounded corners=4pt, inner sep=1pt,right,rotate=270] at (14.5,16.8) {$d=\min\{2x-2, x+y-2\}$ };

\draw [very thick] (1,4) -- (16,19);
\node [rotate=45] at (11,15) { $e-y = d-x+2$ };

\draw [thick] (1,16) -- (18,-1);
\node [fill=white, rounded corners=4pt, inner sep=1pt,right, rotate=-45] at (2,16) { $0$-stem,\: $d+e=x+y$ };

\draw [thick, dashed] (1,15) -- (17,-1);
\node [fill=white, rounded corners=4pt, inner sep=1pt, right, rotate=-45] at (0.2,14.8) { $-1$-stem, $d+e=x+y-1$ };

\draw [thick] (1,8) -- (20,8);
\node [right] at (14.3,7.5) { weight $-1$, \: $e=y-1$ };

\draw [thick] (8,-2) -- (8, 20);
\node [fill=white, rounded corners=4pt, inner sep=1pt, right, rotate=270] at (7.5, 7) {coweight $0$, \quad $d=x$};

\node [fill=white, rounded corners=4pt, inner sep=2pt] at (11,17.5) {\Large finite};
\node [fill=white, rounded corners=4pt, inner sep=2pt] at (11,11) {\Large $\pi_{d+e}(S^{x+y})$};
\node [fill=white, rounded corners=4pt, inner sep=2pt] at (11,1.5) { $\Hoh^{x+y-d-e, y-e}$};
\node [fill=white, rounded corners=4pt, inner sep=2pt] at (12.5,6.5) {\LARGE B--S};
\end{tikzpicture}

\caption{Some facts about $\pi_{d+e\alpha}(S^{x+y\alpha})$ with reference to the comparison $\pi_{d+e\alpha}(S^{x+y\alpha}) \to \pi_{d+e}(S^{x+y})$ for fixed $x,y$. Points $(d,e)$ in which $d < x$ are not marked, since $\pi_{d+e\alpha}(S^{x+y\alpha})$ vanishes for dimensional reasons.}
\label{fig:chart}
\end{figure}


The case of the spheres is a base case for an induction that allows us to say something about many Stiefel varieties. This brings us to our next two main theorems, which we restate for the benefit of the reader.
\surjectivityStiefel*

\begin{proof}
  Let us disregard the last paragraph for the moment. The proof is by induction on $r$. The case of $r=1$ follows from Proposition \ref{prop:comparisonOfUH} with $x=n-1$ and $y=n$, since $V_1(\A^n) \weq S^{n-1+n\alpha}$. The inequalities on $e,d$ imply that $\pi_{d+e}(W_1(\C^n))= \pi_{d+e}(S^{2n-1})$ is a stable homotopy group of a sphere and inequality \ref{ineq:z} implies that it is torsion.

  For the induction step, suppose $r \ge 2$ but $r \le n-2$. We consider the commutative diagram whose rows are exact sequences
{\small  \begin{equation}\label{eq:diagram}
    \begin{tikzcd}[column sep = small]
      \pi_{d+1+e\alpha}(V_1(\A^n)) \rar \dar{f_1} & \pi_{d+e\alpha}(V_{r-1}(\A^{n-1})) \rar \dar{f_2} & \pi_{d+e\alpha}(V_r(\A^n)) \rar \dar{f_3} & \pi_{d+e\alpha}(V_1(\A^n))  \dar{f_4} \rar & \pi_{d-1+e\alpha}(V_{r-1}(\A^{n-1})) \dar{f_5}  \\
      \pi_{d+e+1}(W_1(\C^n)) \rar  & \pi_{d+e}(W_{r-1}(\C^{n-1})) \rar  & \pi_{d+e}(W_r(\C^n)) \rar  & \pi_{d+e}(W_1(\C^n))\rar & \pi_{d+e-1}(W_{r-1}(\C^{n-1})).
    \end{tikzcd}
  \end{equation}}
The commutativity of this diagram is proved in \cite[Prop.\ 5.1]{Gant2025}. Map $f_1$ is surjective with uniquely divisible kernel by Proposition \ref{prop:comparisonOfUH}: the required inequalities may be easily verified with $x=n-1$, $y=n$ and $d+1$ playing the part of $d$. Note that inequalities \ref{ineq:f}--\ref{ineq:z} imply the same inequalities with $r-1$, $n-1$ in place of $r$, $n$. Therefore, the induction hypothesis implies that the map $f_2$ is surjective with uniquely divisible kernel. Map $f_4$ is surjective with uniquely divisible kernel by Proposition \ref{prop:comparisonOfUH} again. And map $f_5$ is a surjection with uniquely divisible kernel by the induction hypothesis again: inequalities \ref{ineq:f}--\ref{ineq:z} imply the same inequalities with $r-1$, $n-1$ and $d-1$ in place of $r,n$ and $d$. Additionally, the codomains of $f_2$ and $f_4$ are torsion abelian groups, by the induction hypothesis. It follows that the codomain of $f_3$ is also torsion.

  It follows from our modified five-lemma, Proposition \ref{pr:5lemmaDiv}, that $f_3$ is surjective with uniquely divisible kernel. Lemma \ref{lem:divTorExt} now assures us that $f_3$ is split. 
  
  Now we return to the question of isomorphism.
  If additionally $n-1 \le e$, then all instances of ``surjective'' may be replaced above with ``isomorphic'', by reference to Proposition \ref{prop:comparisonOfUH}, and the usual five-lemma may be used in place of Proposition \ref{pr:5lemmaDiv}. 
\end{proof}

\injectivityStiefel*

\begin{proof}
  The proof is by induction on $r$, and follows the proof of Theorem \ref{th:surjectivityStiefel}  extremely closely. In the case of $r=1$, the inequalities above coincide with the inequalities of the stronger form of Theorem \ref{th:surjectivityStiefel}.
  
  For the induction step, suppose $r \ge 2$ but $r \le n-2$. We consider the commutative diagram whose rows are exact sequences
  \begin{equation*}
    \begin{tikzcd}
      \pi_{d+1+e\alpha}(V_1(\A^n)) \rar \dar{f_1} & \pi_{d+e\alpha}(V_{r-1}(\A^{n-1})) \rar \dar{f_2} & \pi_{d+e\alpha}(V_r(\A^n)) \rar \dar{f_3} & \pi_{d+e\alpha}(V_1(\A^n))  \dar{f_4} \\
      \pi_{d+1+e}(W_1(\C^n)) \rar  & \pi_{d+e}(W_{r-1}(\C^{n-1})) \rar  & \pi_{d+e}(W_r(\C^n)) \rar  & \pi_{d+e}(W_1(\C^n)).
    \end{tikzcd}
  \end{equation*}
  This differs from diagram \eqref{eq:diagram} by the omission of $f_5$.
  Map $f_1$ is an isomorphism by Proposition \ref{prop:comparisonOfUH}; the required inequalities may be easily verified with $x=n-1$, $y=n$ and $d+1$ playing the part of $d$. Note that inequalities \ref{ineq:fprime}--\ref{ineq:zprime} imply the same inequalities with $r-1$, $n-1$ in place of $r$, $n$. Therefore map $f_2$ is an injection. Map $f_4$ is an isomorphism by Proposition \ref{prop:comparisonOfUH} again. The result now follows by a diagram chase.
\end{proof}

\begin{remark}
    If the Beilinson--Soulé conjecture holds for $k$, then the condition $n-1 \le e$ may be replaced by $\min\{n-1, 2n-d\} \le e$ in each of Theorems \ref{th:surjectivityStiefel} and \ref{th:injectivityStiefel}. The proofs are unchanged; the result on which they rely, Proposition \ref{prop:comparisonOfUH}, is stronger.
\end{remark}

\section{Application to maps of Stiefel varieties and stably free modules}
\label{sec:appl-sect-stief}

For any $r \in \{1,\dots, n\}$, there is a projection map $\rho: V_r(\A^n) \to V_1(\A^n)$. The problem of determining whether $\rho$ has a right inverse (alternatively termed ``a section'' since $\rho$ is a fibre bundle) was studied in \cite{Gant2025}. Using Theorem \ref{th:injectivityStiefel}, we can settle all cases of this problem over $\bar \Q$.

We recall the definition of the James numbers $b_q$, which are also known as the Atiyah--Todd numbers $M_q$. Definitions are given in \cite{James1958}, \cite{Atiyah1960}, and these definitions are proved to be equal in \cite{Adams1965}. Explicitly, they are described by their $p$-adic valuations, $v_p$, for all primes $p$:
\[ v_p(b_q) = \begin{cases} \max\left\{s + v_p(s) \:\Big|\: 1 \le s \le \lfloor \frac{ q-1 }{p-1} \rfloor \right\}, \quad \text{ if  $q \ge p$;}  \\ 
0 \qquad \text{ otherwise.} \end{cases} \]
The first few James numbers may easily be listed:
\[ b_2 = 2; \quad b_3=b_4=2^3 3=24; \quad b_5=2^6 3^2 5=2880; \quad  \]
We remark that the James numbers grow quickly. We establish a rough bound: $v_2(b_q) \ge q-1$, so that $b_q \ge 2^{q-1}$. 

This brings us to our last main theorem, which we restate for the benefit of the reader.
\completeSplittingResult*
\begin{proof}
  If $b_r \nmid n$, then there cannot be a right inverse by virtue of \cite[Thm.~6.5]{Raynaud1968}. Note that this result disallows right inverses even in the $\Aone$-homotopy category $\cat{H}(\bar \Q)$.

  Let us therefore suppose that $b_r \mid n$, and try to construct a right inverse. The case of $r=2$ follows from a well known construction of Bass, first published to our knowledge in \cite[Prop.\ 2.2(b)]{Raynaud1968} as a construction of a free summand in a projective module.

  We therefore may suppose $r \ge 3$. By using the rough bound we established earlier, we deduce that $r \le \log_2(n) +1$. If $n \ge 9$, this implies that $r < n/2$. For $n \in \{5,6,7,8\}$, the inequality $b_r \le n$ implies $b_r=2$, so $r< n/2$ in this case as well.  

  Theorem \ref{th:injectivityStiefel} applies with $(n-2,n)$ playing the part of $(d,e)$ and $(r-1,n-1)$ playing the part of $(r,n)$. Inequality \ref{ineq:needsrn2} of Theorem \ref{th:injectivityStiefel} requires $r < n/2$ in this case. The conclusion of the theorem is that the realization map
  \[\pi_{n-2+n\alpha}(V_{r-1}(\A^{n-1}_{\bar \Q})) \to \pi_{2n-2}(W_{r-1}(\C^{n-1})) \]
  is injective. We now use \cite[Prop.~7.3]{Gant2025}, which tells us that $\rho$ has a right inverse if  the corresponding map of complex Stiefel manifolds
  \[ \rho(\C):W_r(\C^n) \to W_1(\C^n) \]
  has a right inverse. This has a right inverse by \cite[Thm.~1.1]{Adams1965}.
\end{proof}

\appendix

\section{Homological algebra modulo uniquely divisible subgroups}
\label{sec:homol-algebra-modulo}

\begin{lemma} \label{lem:2of3divisibility}
  If
  \[ 0 \to A \to B \to C \to 0 \]
  is a short exact sequence of abelian groups with the property that two of the groups $A,B,C$ are uniquely divisible, then so is the third.
\end{lemma}
\begin{proof}
    By definition, an abelian group $M$ is uniquely divisible if the self-maps $\times n : M \to M$ are isomorphisms for all natural numbers $n$. The result follows easily from the snake lemma.
\end{proof}

\begin{proposition}\label{pr:5lemmaDiv}
  Consider a commutative diagram of abelian groups
  \begin{equation}\label{eq:5lemma}
    \begin{tikzcd}
      A_1 \arrow[d,two heads,"\phi_1"] \rar & A_2 \arrow[d,two heads,"\phi_2"] \rar & A_3 \arrow[d,"\phi_3"] \rar & A_4 \arrow[d,two heads,"\phi_4"] \rar & A_5 \arrow[d,two heads,"\phi_5"] \\
      B_1 \rar & B_2 \rar & B_3 \rar & B_4 \rar & B_5,
    \end{tikzcd}
  \end{equation}
  in which the rows are exact sequences.
  Suppose that $\phi_1, \phi_2, \phi_3, \phi_4$ are surjective and their kernels are uniquely divisible. Suppose further that $B_3$ is torsion. Then $\phi_3$ is surjective and $\ker(\phi_3)$ is uniquely divisible.
\end{proposition}
\begin{proof}
  Treat diagram \eqref{eq:5lemma} as a double complex where $A_i$ is in degree $(i,0)$ and $B_i$ is in degree $(i,1)$. There are two spectral sequences calculating the homology of the total complex: for each of these, the nonzero groups on the $E_0$-page coincide with the groups in \eqref{eq:5lemma}. In one, the $d_0$ differentials are the horizontal arrows in \eqref{eq:5lemma}, and since the rows are exact, we see that the homology of the total complex vanishes in total degrees $3$ and $4$.

  The other spectral sequence has as $E_0$-page the following diagram
   \begin{equation*}
    \begin{tikzcd}
      A_1 \dar{\phi_1}  & A_2 \dar{\phi_2}  & A_3 \dar{\phi_3}  & A_4 \dar{\phi_4}  & A_5 \dar{\phi_5} \\
      B_1  & B_2  & B_3  & B_4 &  B_5,
    \end{tikzcd}
  \end{equation*}
  so that the $E_1$-page is
  \begin{equation}
    \begin{tikzcd}\label{eq:kernels}
      \ker(\phi_1) \rar & \ker(\phi_2) \rar & \ker(\phi_3) \rar & \ker(\phi_4) \rar & \ker(\phi_5) \\
      0 \rar & 0 \rar & \coker(\phi_3) \rar & 0 \rar  & 0.
    \end{tikzcd}
  \end{equation}
  On the $E_2$-page and later, there are no nonzero differentials arriving at or emanating from the $(3,0)$-group, which must be $0$ by the $E_\infty$-page, so that we deduce that the upper sequence in \eqref{eq:kernels} is exact at $\ker(\phi_3)$. The terms in this sequence, other than $\ker(\phi_3)$, are uniquely divisible by hypothesis, so that the five-lemma implies that $\ker(\phi_3)$ is also uniquely divisible.

  The $E_2$-page takes the form
  \[
  \begin{tikzcd}\label{eq:E2}
     \ast  & \ast   & 0  & 0  & Q \\
      0 \arrow[urr] & 0 \arrow[urr] & \coker(\phi_3) \arrow[urr, hook]  & 0   & 0.
    \end{tikzcd} 
  \]
  where $Q$ is the cokernel of $\ker(\phi_4) \to \ker(\phi_5)$ and $\ast$ denotes some unspecified abelian groups. For dimensional reasons, there are no further nonzero differentials, from which it follows that $\coker(\phi_3) \to Q$ must be an injection: the abutment of the sequence is $0$ in total degree $4$.

  Since $\ker(\phi_4)$ is uniquely divisible, its image in $\ker(\phi_5)$ is divisible, and therefore uniquely divisible. It follows from Lemma \ref{lem:2of3divisibility} that $Q$ is uniquely divisible. In particular, the map $\coker(\phi_3) \to Q$, which has a torsion domain, must be $0$. Since it is also an injection, we deduce that $\coker(\phi_3)$ is $0$ as required.
\end{proof}

\printbibliography
\end{document}
